\documentclass[12pt, reqno]{amsart}
\usepackage{amsmath, amsthm, amscd, amsfonts, amssymb, graphicx, color}
\usepackage[bookmarksnumbered, colorlinks, plainpages]{hyperref}
\hypersetup{colorlinks=true,linkcolor=red, anchorcolor=green, citecolor=cyan, urlcolor=red, filecolor=magenta, pdftoolbar=true}

\newtheorem{theorem}{Theorem}[section]

\newtheorem{proposition}[theorem]{Proposition}

\theoremstyle{definition}
\newtheorem{definition}[theorem]{Definition}
\newtheorem{example}[theorem]{Example}

\theoremstyle{remark}
\newtheorem{remark}[theorem]{Remark}
\numberwithin{equation}{section}

\newcommand{\tr}{{\rm Tr\hskip -0.2em}~}

\newcommand{\mean}{{\rm E}}

\begin{document}

\setcounter{page}{1}

\title[Expectation of means]{An Inequality for Expectation of Means of Positive Random Variables}

\author[P. Gibilisco \MakeLowercase{and} F. Hansen]{Paolo Gibilisco$^1$$^{*}$ \MakeLowercase{and} Frank Hansen$^2$}

\address{$^{1}$Department of Economics and Finance, University of Rome ``Tor Vergata", Via Columbia 2, Rome 00133, Italy.}
\email{\textcolor[rgb]{0.00,0.00,0.84}{gibilisco@volterra.uniroma2.it}}

\address{$^{2}$Institute for Excellence in Higher Education, Tohoku University, Sendai, Japan.
}
\email{\textcolor[rgb]{0.00,0.00,0.84}{frank.hansen@m.tohoku.ac.jp}}






\begin{abstract}
Suppose that $X,Y$ are positive random variable and $m$ a numerical (commutative) mean. We prove that the inequality $\mean (m(X,Y)) \leq m(\mean (X), \mean (Y))$ holds if and only if the mean is generated by a concave function. With due changes we also prove that the same inequality holds for all operator means in the Kubo-Ando setting. The case of the harmonic mean was proved  by C.R. Rao and B.L.S. Prakasa Rao.

\end{abstract} \maketitle

\section{Introduction and preliminaries}

Let $x,y$ be positive real numbers. The arithmetic, geometric, harmonic, and logarithmic means are defined by
\[
\begin{array}{rlrl}
m_a(x,y)&=\displaystyle\frac{x+y}{2} &m_g(x,y)&=\sqrt{xy}\\[2ex]
m_h(x,y)&=\displaystyle\frac{2}{x^{-1}+y^{-1}} \qquad&m_l(x,y)&=\displaystyle\frac{x-y}{\log x - \log y} \,.
\end{array}
\]
Suppose $ X,Y\colon\Omega \to (0,+\infty)$ are positive random variables. Linearity of the expectation operator trivially implies
\[
\mean(m_a(X,Y))=m_a(\mean(X),\mean (Y)).
\] 
On the other hand the Cauchy-Schwartz inequality implies
\[
\mean(m_g(X,Y)) \leq m_g(\mean(X),\mean(Y)).
\] 
Working on a result by Fisher on ancillary statistics
Rao \cite{Rao:1992, Rao:1996}  obtained the following proposition by an application of H\"older's inequality together with the harmonic-geometric mean inequality.

\begin{proposition}
\begin{equation} \label{mainharmonic}
\mean (m_h(X,Y)) \leq m_h(\mean (X), \mean (Y)). 
\end{equation}
\end{proposition}

It is natural to ask about the generality of this result. For example, does it hold also for the logarithmic mean? To properly answer this question it is better to choose one of the many axiomatic approaches to the notion of a mean.  

In Section \ref{commutativeperspective} we recall the notion of {\em perspective} of a function, and in Section \ref{commutativemeans} we recall that a mean of pairs of positive numbers may be represented as the perspective of a certain representing function.  In Section \ref{commutativemain} we prove that inequality (\ref{mainharmonic}) holds for a mean $m_f$ if and only if the representing function $f$ is concave. 

Once this is done it becomes natural to address the analog question in the non-commuta\-ti\-ve setting. A positive answer to the case of the matrix harmonic mean was given by Prakasa Rao in \cite{PrakasaRao:1998} and by C.R. Rao in \cite{Rao:2000}. But also in this case the inequality holds in a much wider generality.
In Section \ref{noncommutativeperspective} we recall the notion of non-commutative perspectives and 
some of their properties, while in
Section \ref{KuboAndomeans} we describe the subclass of Kubo-Ando operator means. In Section \ref{noncommutativemain} we show that inequality (\ref{mainharmonic}) holds true also in the non-commutative case. This follows from the fact that operator means are generated by operator monotone functions; indeed operator monotonicity of a function defined in the positive half-line implies operator concavity \cite[Corollary 2.2]{kn:hansen:2013:1}; rendering the non-commutative setting completely different from the commutative counter part. 

In Section \ref{randommatrix} we consider the random matrix case which, to some extent, encompasses the previous results.

\section{Perspective of a function: commutative case} \label{commutativeperspective}

Let $K \subseteq {\mathbb R}^n$ be a non-empty convex set, and let $g\colon K \to \mathbb R$ be a function. We consider the set
\[
L=\{ (x,t) \mid t>0,\, t^{-1}x \in K     \}.
\]

\begin{definition}
The perspective ${\mathcal P}_g $ of $g$ is the function $\mathcal P_g\colon L \to \mathbb R$ defined by setting
\[
\mathcal P_g(x,t)=t g(t^{-1} x)\qquad (x,t)\in L.
\]
\end{definition}

The following classical result is well-known.

\begin{proposition}\label{jointconvexity}
The perspective $\mathcal P_g $ of a convex function $ g $ is convex.
\end{proposition}

\begin{example}
Consider the convex funtion
\[
g(x)=x \log x \qquad x>0
\]
with limit $ g(0)=0 $ and set $ K=(0, \infty). $ Then the perspective is the relative entropy
\[
\mathcal P_g(x,t)=x \log x - x \log t
\]
for $ x,t>0. $

\end{example}

Notice that the perspective of a concave function is concave.

\section{Means for positive numbers}\label{commutativemeans}

We use the notation ${\mathbb R}_+= (0,+\infty)$. 

\begin{definition} \label{numbermean}

A bivariate {\sl mean} \cite{PetzTemesi:2005}  is a function $m\colon{\mathbb R}_+ \times {\mathbb R}_+ \to{\mathbb R}_+$  such that

\begin{enumerate}

\item $m(x,x)=x.$ 

\item $m(x,y)=m(y,x).$ 

\item $x <y  $ $\,\Rightarrow\,$ $ x<m(x,y)<y.$ 

\item $x<x' $ and $ y<y' $ $\,\Rightarrow\,$ $ m(x,y)<m(x',y'). $ 

\item $m$ is continuous.

\item $ m $ is positively homogeneous; that is $m(tx,ty)=t \cdot m(x,y)$ for $ t>0. $

\end{enumerate}

\end{definition}

We use the notation $ \mathcal M_{num} $ for the set of bivariate means described above.

\begin{definition}
Let $ \mathcal F_{num} $ denote the class of functions $f\colon\mathbb R_+ \to\mathbb R_+$ such that 

\begin{enumerate}

\item $f$ is continuous.

\item $f$ is monotone increasing.

\item $f(1)=1.$

\item $tf(t^{-1})=f(t)$ for $ t>0. $

\end{enumerate}
\end{definition}

The following result is straightforward.

\begin{proposition}
There is bijection betwen ${\mathcal M}_{num}$ and ${\mathcal F}_{num}$ given by the formulas
\[
m_f(x,y)=yf(y^{-1}x)\qquad\text{and}\qquad f_m(t)=m(1,t)
\]
for positive numbers $ x,y $ and $ t. $
\end{proposition}

\subsection{Some examples of means}

The functions in the table below are all concave, even operator concave.

\smallskip 

\begin{table}[ht]
\caption{}\label{eqtable}
\renewcommand\arraystretch{2.5}
\renewcommand{\tabcolsep}{8pt}
\noindent\[
\begin{array}{|c|c|c|}
\hline
{\rm Name} & {\rm function} & {\rm mean}\\
\hline 
{\rm arithmetic} & \displaystyle\frac{1+x}{2} & \displaystyle\frac{x+y}{2}\\
\hline  
{\rm WYD}, \beta\in(0,1)  & \displaystyle\frac{x^{\beta}+x^{1-\beta}}{2} & \displaystyle\frac{x^{\beta}y^{1-\beta}+x^{1-\beta}y^{\beta}}{2} \\
\hline
{\rm geometric} & \sqrt{x} & \sqrt{xy} \\
\hline
{\rm harmonic} & \displaystyle\frac{2x}{x+1} & \displaystyle\frac{2}{x^{-1}+y^{-1}} \\
\hline
{\rm logarithmic} & \displaystyle\frac{x-1}{\log x} & \displaystyle\frac{x-y}{\log x - \log y} \\
\hline
\end{array}
\]
\end{table}

However, there exist non-concave functions in ${\mathcal F}_{num}.$ Consider for example the function
\[
g(x)=\frac{1}{4}\begin{cases} x+3 &  0 \leq x \leq1,\\[0.5ex]
\displaystyle 3x+1 &  x \geq 1.
\end{cases}
\]
This piece-wise affine function is convex and belongs to  $ \mathcal F_{num}. $

\section{The main result: commutative case}\label{commutativemain}

\begin{theorem} Take a function $f \in \mathcal F_{num}.$ 
The inequality
\begin{equation} \label{Conj1}
\mean(m_f(X,Y)) \leq m_f(\mean(X), \mean(Y))
\end{equation}
holds for arbitrary positive random variables $X$ and $Y$  if and only if $ f $ is concave.
\end{theorem}

\begin{proof}
Suppose inequality (\ref{Conj1}) holds for a function $f.$ Take $\Omega=\{1,2\}$ as state space with probabilities $ p $ and $ 1-p, $ and let $Y$ be the constant function $ 1. $ We set  $X(1)=x_1 $ and  $X(2)=x_2 $ for given $ x_1,x_2>0. $ We then have $\mean (Y)=1$ and thus
\[
m_f\bigl(\mean(X), \mean(Y)\bigr)= \mean(Y) f\kern-3pt\left(\frac{\mean(X)}{\mean(Y)}\right)=f(p x_1+(1-p)x_2).
\]
We also have
\[
m_f(X,Y)(1)=Y(1)f\kern-3pt\left(\frac{X(1)}{Y(1)}\right)=f(x_1)
\]
and
\[
m_f(X,Y)(2)=Y(2)f\kern-3pt\left(\frac{X(2)}{Y(2)}\right)=f(x_2).
\]
Therefore
\[
\begin{array}{rl}
p f(x_1)+(1-p)f(x_2)&=\mean\bigl(m_f(X,Y)\bigr)\leq m_f(\mean(X), \mean(Y))\\[2ex]
&= f(p x_1+(1-p) x_2)
\end{array}
\]
implying that $ f $ is concave.

Suppose on the other hand that $f$ is concave and consider two positive random variables $X$ and $Y$. We only have to prove the theorem under the assumption that $X$ and $Y$ are simple random variables (finite linear combinations of indicator functions). The general case then follows since any positive random variable is a pointwise increasing limit of simple random variables.
The (different) values of $ X $ are denoted by $ x_1,\dots,x_n $ with associated (marginal or unconditional) probabilities $ p_1,\dots,p_n. $ The (different) values of $ Y $ are denoted by $ y_1,\dots,y_m $ with associated (marginal or unconditional) probabilities $ q_1,\dots,q_m. $ 

The stochastic variable
$ m_f(X,Y) $ takes the values $ m_f(x_i,y_j) $ with probabilities $ P(X=x_i\,, Y=y_j) $ for $ i=1,\dots,n $ and $ j=1,\dots,m $ (possibly counted with multiplicity).
The mean $m_f$ is the perspective of $f$ and thus concave by Proposition~\ref{jointconvexity}. We may therefore apply Jensen's inequality and obtain
\[
\begin{array}{l}
\mean\bigl(m_f(X,Y)\bigr)=\displaystyle\sum_{i=1}^n\sum_{j=1}^m P(X=x_i\,, Y=y_j) m_f(x_i, y_j)\\[3ex]
\le\displaystyle m_f\left(\sum_{i=1}^n\sum_{j=1}^m P(X=x_i\,, Y=y_j) (x_i, y_j)\right)\\[3ex]
=\displaystyle m_f\left(\sum_{i=1}^n\sum_{j=1}^m P(X=x_i\,, Y=y_j) x_i, \sum_{j=1}^m\sum_{i=1}^n P(X=x_i\,, Y=y_j) y_j\right),
\end{array}
\]
where we interchanged the summations in the second argument of $ m_f. $ Since the sums of the joint probabilities
\[
\sum_{j=1}^m P(X=x_i\,, Y=y_j)=p_i\quad\text{and}\quad\sum_{i=1}^n P(X=x_i\,, Y=y_j)=q_j
\]
we obtain
\[
\mean\bigl(m_f(X,Y)\bigr)\le m_f\left(\sum_{i=1}^n p_i x_i\,, \sum_{j=1}^m q_j y_j\right)= m_f\bigl(\mean(X),\mean(Y)\bigr),
\]
which is the desired inequality (\ref{Conj1}). 
\end{proof}

\section{Non-commutative perspective} \label{noncommutativeperspective}

For the basic results of this section we refer to \cite{Effros:2009, ENG:2011, EffrosHansen:2014}.
Let $f$ be a function defined in the open positive half-line. In Section \ref{commutativeperspective} we recalled
the perspective of $ f $ as the function of two variables $\mathcal P_f(t,s) = s f(s^{-1}t),$  where  $t,s > 0.$
Depending on the application, we may also consider the function $(t,s)\to\mathcal P_f (s, t)$ and denote this as the perspective of $f$.

If $A$ and $B$ are commuting positive definite matrices, then the matrix ${\mathcal P}_f (A, B)$ is well-defined by the functional calculus, and it coincides with $Bf(B^{-1}A).$ However,
even if $A$ and $B$ do not commute one may, by choosing an appropriate ordering, define the perspective.

\begin{definition} Let $ f $ be a function defined in the open positive half-line. The (non-commutative) perspective $ \mathcal P_f $ of $ f $ is then defined by setting
\[
{\mathcal P}_f(A,B)=A^{1/2}f(A^{-1/2} B A^{-1/2})A^{1/2}
\]
for positive definite operators $ A $ and $ B. $
\end{definition}

For the following basic result confer \cite[Theorem 2.2]{Effros:2009},  
\cite[Theorem 1.1]{EffrosHansen:2014} and 
\cite[Theorem 2.2]{ENG:2011}.

\begin{theorem}\label{BasicJointConvexityNC}
The (non-commutative) perspective ${\mathcal P}_f $ is convex if and only if $f$ is operator convex.
\end{theorem}

Let $ f\colon(0,\infty)\to\mathbf R $ be a convex function. Since the perspective
$  \mathcal P_f $ is both convex and positively homogenous we obtain the inequality
\[
\mathcal P_f\Big(\sum_{i=1}^n \lambda_i x_i,\sum_{i=1}^n \lambda_i y_i\Bigr)
\le\sum_{i=1}^n \lambda_i \mathcal P_f(x_i,y_i)
\]
for tuples $ (x_1,\dots,x_n) $ and $ (y_1,\dots,y_n) $ of positive numbers and
positive numbers $ \lambda_1,\dots,\lambda_n. $ This entails, by setting all $ \lambda_i=1, $ the inequality
\[
\mathcal P_f(\tr A,\tr B)\le \tr \mathcal P_f(A,B)
\]
for commuting positive definite matrices $ A $ and $ B. $

The transformer inequality  for the non-commutative perspective of an operator convex function is essentially proved in \cite[Theorem 2.2]{Hansen:1983}. Since the perspective of an operator convex function is a convex regular operator map the statement also follows from \cite[Lemma 2.1]{kn:hansen:2016:1}.

\begin{proposition}[the transformer inequality]
Let $ f\colon(0,\infty)\to\mathbb R $ be an operator convex function. The non-commutative perspective $ \mathcal P_f $ satisfies the inequality
\[
\mathcal P_f(C^*AC, C^*BC)\le C^*\mathcal P_f(A,B) C
\]
for every contraction $ C $ and positive definite operators $ A $ and $ B. $
\end{proposition}

 Notice that we by homogeneity obtain
 \[
\mathcal P_f(C^*AC, C^* BC)\le  C^* \mathcal P_f(A,B) C
 \]
 for any operator $ C. $ In particular, if $ C $ is invertible we then have
 \[
 \mathcal P_f(A,B)\le (C^*)^{-1} \mathcal P_f(C^*AC, C^* BC) C^{-1}\le \mathcal P_f(A,B),
 \]
 hence there is equality and thus
 \begin{equation}
 C^* \mathcal P_f(A,B) C= \mathcal P_f(C^*AC, C^* BC).
 \end{equation}

\begin{proposition}
Let $ \mathcal P_f $ be the non-commutative perspective of an operator convex function $ f\colon (0,\infty)\to\mathbb R $ and let $ c_1,\dots,c_n $ be operators on a Hilbert space $ \mathcal H $ such that
$ c_1^*c_1+\cdots+c^*_n c_n= 1. $ Then
\[
\mathcal P_f\Bigl(\sum_{i=1}^n c_i^* A_i c_i\,, \sum_{i=1}^n c_i^* B_i c_i\Bigr)\le
\sum_{i=1}^n c_i^* \mathcal P_f(A_i,B_i) c_i
\]
for positive definite operators $ A_1,\dots,A_n $ and $ B_1,\dots, B_n $ acting on $ \mathcal H. $ 
\end{proposition}

\begin{proof}
The perspective $ \mathcal P_f $ is a convex regular operator map of two variables 
\cite{Hansen:1983, EffrosHansen:2014, kn:hansen:2016:1}. The statement thus follows from Jensen's inequality for convex regular operator maps  \cite[Theorem 2.2]{kn:hansen:2016:1}.
\end{proof}

\section{Operator means in the sense of Kubo-Ando} \label{KuboAndomeans}

The celebrated Kubo-Ando theory of matrix means 
\cite{KuboAndo79/80, PetzTemesi:2005,GibiliscoHansenIsola:2009} may today be considered as part of the theory of perpectives of positive operator concave functions. This setting is simpler than the general theory of perspectives since a positive operator concave function necessarily is increasing, while a positive operator convex function may not necessarily be monotonic.  

\begin{definition}
A bivariate {\sl mean} for pairs of positive operators is a function
\[
(A,B)\to m(A,B)
\]
defined in and with values in positive definite operators on a Hilbert space and satisfying, mutatis mutandis, conditions $(1)$ to $(5)$ in Definition~\ref{numbermean}. In addition the
{\sl transformer inequality}
\[
C^*m(A,B)C \leq m(C^*AC,C^*BC)
\]
holds for positive definite $ A, B $ and arbitrary $ C. $ 
\end{definition}

Notice that the transformer inequality replaces $ (6) $ in Definition~\ref{numbermean}.
We denote by $\displaystyle {\mathcal M}_{op}$ the set of matrix means.

\begin{example}
The  arithmetic, geometric and harmonic (matrix) means are defined, respectively, by setting
\[
\begin{array}{rcl}
A \nabla B&=&\frac{1}{2}(A+B)\\[1.5ex]
A\# B&=&A^{1/2}\bigl(A^{-1/2} B A^{-1/2}\bigr)^{1/2}A^{1/2}\\[1.5ex]
A{\rm !}B&=&2(A^{-1}+B^{-1})^{-1}. 
\end{array}
\]
\end{example}

We recall that a function $f\colon(0,\infty)\to \mathbb{R}$ is said to be 
{\it operator monotone (increasing)} if
\[
A\le B\quad\Rightarrow\quad f(A)\le f(B)
\]
for positive definite operators on an arbitrary Hilbert space. An operator monotone function $ f $ is said to be {\it symmetric} if
$f(t)=tf(t^{-1})$ for $ t>0 $ and {\it normalized} if $f(1)=1.$

\begin{definition}

${\mathcal F}_{op}$ is the class of functions $f: {\mathbb R}_+
\to{\mathbb R}_+$ such that
\begin{enumerate}

\item $f$ is operator monotone increasing,

\item $tf(t^{-1})=f(t)\qquad t>0,$

\item $f(1)=1.$

\end{enumerate}
\end{definition}

The fundamental result, due to Kubo and Ando, is the following.

\begin{theorem}
There is bijection between ${\mathcal M}_{op}$ and ${\mathcal F}_{op}$ given by
the formula
\[
m_f(A,B)= A^{1/2}f(A^{-1/2} BA^{-1/2})A^{1/2}.
\]
\end{theorem}

\begin{remark}
All the function in ${\mathcal F}_{op}$ are (operator) concave making the operator case quite different from the numerical one.
\end{remark}

If $\rho$ is a density matrix and $A$ is self-adjoint then the expectation of $ A $ in the state $ \rho $ is defined by setting $ \mean_{\rho}(A)= \tr(\rho A). $

\section{The main result: noncommutative case}\label{noncommutativemain}

\begin{theorem}\label{non-commutative Rao inequality}
Take $ f\in {\mathcal F}_{op}. $ Then
\begin{equation} \label{Conj2}
\mathbb {\mean}_\rho(m_f(A,B)) \leq m_f({\mean}_\rho(A), {\mean }_\rho(B)),
\end{equation}

\end{theorem}

\begin{proof} Consider a spectral resolution
\[
\rho=\sum_{i=1}^n \lambda_i e_i
\]
of the density matrix $ \rho $ in terms of one-dimensional orthogonal eigenprojections $ e_1,\dots,e_n $ with corresponding eigenvalues $ \lambda_1,\dots,\lambda_n $ counted with multiplicity. By setting $ c_i=\lambda_i^{1/2} e_i $ for $ i=1,\dots,n $ we obtain
\[
\mean_\rho(A)=\tr\rho A=\tr\sum_{i=1}^n c_i^* A c_i
\]
for any operator $ A. $ By using the transformer inequality we obtain
\[
\begin{array}{rl}
\displaystyle\mean_\rho\bigl(m_f(A,B)\bigr)&=\displaystyle\tr\sum_{i=1}^n c_i^* m_f\bigl(A,B\bigr)c_i\\[3ex]
&\le\displaystyle\tr m_f\Bigl(\sum_{i=1}^n c_i^*Ac_i,\sum_{i=1}^n c_i^* Bc_i\Bigr)\\[3ex]
&\displaystyle\le m_f\Bigl(\tr\sum_{i=1}^n c_i^*Ac_i\,,\tr\sum_{i=1}^n c_i^*Bc_i\Bigr)\\[3.5ex]
&=m_f\bigl(\mean_\rho(A),\mean_\rho(B)\bigr),
\end{array}
\]
where we in the second inequality used that the operators
\[
\sum_{i=1}^n c_i^*Ac_i\quad\text{and}\quad \sum_{i=1}^n c_i^*Bc_i
\]
are commuting.
\end{proof}

\section{The random matrix case}\label{randommatrix}

Let $(\Omega, {\mathcal F}, P)$ be a probability space. A map $X\colon\Omega \to M_n$ is called a random matrix. We may write
\[
 X =\big(X_{i,j}\bigr)_{i,j=1}^n \colon\Omega\to M_n
\] 
and say that $ X $ is a positive definite random matrix if
\[
X(\omega) =\big(X_{i,j}(\omega)\bigr)_{i,j=1}^n
\] 
is positive definite for $ P $-almost all $ \omega\in\Omega. $  We may readily consider other types of definiteness for random matrices.

\begin{definition}
A positive semi-definite random matrix $ \rho\colon\Omega\to M_n $ is called a random density matrix if $ \tr\rho=1 $ for $ P $-almost all $ \omega\in\Omega. $
\end{definition}

Let $ X $ and $ \rho $ be random matrices on the probability space $(\Omega, {\mathcal F}, P)$ and suppose that $ \rho $ is a random density matrix.  We introduce the pointwise expectation $ {\mean}_\rho(X) $ by setting
\[
({\mean}_\rho X)(\omega)=\tr \rho(\omega)X(\omega)\qquad\omega\in\Omega.
\]
The pointwise expectation $ {\mean}_\rho(X) $ is a random variable with mean
\[
{\mean} \bigl({\mean}_\rho(X)\bigr)=\int_\Omega \tr\rho(\omega)X(\omega)\,dP(\omega).
\]
If $ \rho $ is a constant density matrix then
\[
{\mean}\bigl({\mean}_\rho(X)\bigr)=\tr\rho\int_\Omega X(\omega)\,dP(\omega)
=\tr\rho{\mean}(X)={\mean}_\rho\bigl({\mean}(X)\bigr),
\]
where $ {\mean}(X) $ is the constant matrix with entries
\[
{\mean}(X)_{i,j}=\int_\Omega X_{i,j}(\omega)\,dP(\omega)\qquad i,j=1,\dots,n.
\]

\begin{theorem}\label{Rao inequality for random matrices}
Let  $ X $ and $ Y $ be positive definite random matrices on a probability space $(\Omega, {\mathcal F}, P).$ For $ f\in {\mathcal F}_{op} $ we obtain the inequality
\[
{\mean}\, {\mean}_\rho(m_f(X,Y)) \leq m_f({\mean} \,{\mean}_\rho(X), {\mean}\, {\mean}_\rho(Y))
\]
for each random density matrix $ \rho $ on $(\Omega, {\mathcal F}, P).$
\end{theorem}

\begin{proof}
The matrices $ X(\omega), $ $ Y(\omega) $ and $ \rho(\omega) $ are  positive definite and $ \rho(\omega) $ has unit trace for almost all $ \omega\in\Omega $. The inequality between random variables 
\[
{\mean}_{\rho(\omega)}\bigl(m_f(X(\omega),Y(\omega))\bigr)
\le
m_f\bigl({\mean}_{\rho(\omega)}(X(\omega)), {\mean}_{\rho(\omega)}(Y(\omega))
\]
is therefore valid by our non-commutative  
inequality in Theorem~\ref{non-commutative Rao inequality}. In particular, by taking the mean on both sides, we obtain
\[
\begin{array}{rl}
{\mean} \, {\mean}_{\rho}\bigl(m_f(X,Y)\bigr)
&\le
{\mean}\bigl(m_f\bigl({\mean}_{\rho}(X), {\mean}_{\rho}(Y)\bigr)\\[1ex]
&\le m_f\bigl({\mean}{\mean}_\rho(X), {\mean} {\mean}_\rho(Y)\bigr),
\end{array}
\]
where we used, in the last inequality, the commutative 
inequality in Theorem \ref{Conj1}.
\end{proof}

Notice that Theorem~\ref{Rao inequality for random matrices} reduces to the non-commutative 
inequality when $ \Omega $ is a one point space, and to the commutative 
inequality when $ n=1. $ If $ \rho $ is a constant matrix then the order of ${\mean}$ and $ {\mean}_\rho $ in the inequality may be reversed.

{\bf Acknowledgments.}  It is a pleasure for the first author to thank Fumio Hiai for discussions and hints on the subject.
The second author acknowledges support from the Japanese government Grant-in-Aid for scientific research 26400104.

\bibliographystyle{amsplain}

\end{document}